\documentclass{amsart}[11pt]
\usepackage{graphicx}
\usepackage{amssymb}
\usepackage{epstopdf}
\usepackage{amsmath,amsthm,amscd,amssymb}
\usepackage{latexsym}
\usepackage[colorlinks,citecolor=red,pagebackref,hypertexnames=false]{hyperref}
\usepackage{geometry}                % See geometry.pdf to learn the layout options. There are lots
\usepackage{algorithm2e}

\DeclareMathOperator*{\argmin}{arg\,min}

\numberwithin{equation}{section}

\theoremstyle{plain}
\newtheorem{theorem}{Theorem}[section]
\newtheorem{lemma}[theorem]{Lemma}
\newtheorem{corollary}[theorem]{Corollary}
\newtheorem{proposition}[theorem]{Proposition}

\theoremstyle{definition}

\newtheorem{example}[theorem]{Example}

\newtheorem{case[theorem]}{Case}

\theoremstyle{remark}
\newtheorem{remark}[theorem]{Remark}

\numberwithin{equation}{section}

\begin{document}

 \title{Non-convex optimization via strongly convex majoirziation-minimization}

%  \author{Baoqing Ding} 
 
  %\author{Nick Feng} 

 \author{Azita Mayeli} 
 %\address{Department of Mathematics, The City University of New York \\ The Graduate Center and Queensborough, New York, NY \\ 10016, USA}
%\email{amayeli@gc.cuny.edu}

 %\author{Ivan Selesnick} 
 %\address{Department of 
%Electrical and Computer Engineering \\ Tandon School of Engineering \\ NYU \\ 
%2 Metrotech Center \\
%Brooklyn, 
%New York 11201, USA }
 %\email{selesi@nyu.edu}

 %Azita Mayeli 
 %\thanks{The research for this paper was supported by the }
%\footnotesize\texttt{{mayeli@ma.tum.de}}
%\thanks{}
%  }
  
   \date{\today}  
   
 \begin{abstract} In this paper, we  introduce a  class of nonsmooth nonconvex least square  optimization problem using convex analysis tools and 
 we    propose to use the  iterative minimization-majorization (MM) algorithm  on a convex set with initializer away from the origin to find an optimal point for the optimization problem. For this,  first we use an approach to construct a class 
 of convex majorizers  which approximate the value of  non-convex cost function on a convex set.  The convergence of the  iterative  algorithm is guaranteed when the initial point $x^{(0)}$ is away from the origin and the iterative points $x^{(k)}$ are obtained in a ball centred at $x^{(k-1)}$ with small radius.   The algorithm converges to a stationary point of cost function when the surregators are strongly convex. For the class of our  optimization problems, the proposed penalizer of the cost function is 
  the   difference of $\ell_1$-norm and  the Moreau envelope of a convex function, and     it is a generalization of  GMC non-separable penalty  function  previously   introduced by Ivan Selesnick in \cite{IS17}. 
 \end{abstract} 
  
  \maketitle
  
 
 \vskip.124in 
 

\section{introduction} 
 Consider the following optimization problem 
 \begin{align}\label{opt problem}
 \min_{x\in {\mathcal C}} F(x)
 \end{align}
 where ${\mathcal C}$ is a closed convex subset of $\Bbb R^N$ and $F: \Bbb R^N\to \Bbb R$ is  a real valued  objective or cost function.  In general $F$ is   continuous   but not  convex nor smooth. Most optimization problems rely heavily on convexity condition of the function $F$ and the lack of convexity  for $F$   makes it usually an NP hard problem to find a global minimum point for the optimization problem (\ref{opt problem}). 
  The convexity condition is in  particular  useful in some practical problems such as in 
image reconstruction and sparse recovery. 
In the absence of the convexity condition, majorization-minimization  (MM)  algorithm has been proved to be a useful tool  in finding local minimization vectors or signals. 
  This algorithm  is an iterative algorithm and it converts a difficult optimization problem  into a simple one, as we will demonstrate some cases in this paper.  % For example, this can be done by substituting a non-convex complicated objective  function $F$ by a simple convex function, and then approximating the minimization point of $F$ at each iteration step. 

   \vskip.124in 
The goal of this paper is for given $y\in \Bbb R^M$ to   solve the   following class of   least square problems 
\begin{align}\label{objective function}
\argmin_{x\in \Bbb R^N} F(x), \quad   \text{where} \ 
 \ 
F(x) = \cfrac{1}{2}\|y-Ax\|_2^2 + \lambda(\|x\|_1 - f_\alpha(x)), \quad \forall  x\in\Bbb R^N, 
\end{align}
using an iterative algorithm. Here, 
 $f_\alpha$   is the Moreau envelope of a  convex function   as  defined in (\ref{Mourea envelop}),  and $\alpha>0, \lambda>0$ are constants and predetermined.  $A\in \Bbb R^M\times\Bbb R^N$ is a low rank  wide matrix (e.g.  a finite frame or  wavelet).  In our setting, 
 we use tools from convex analysis to introduce the new class of non-convex penalties: \begin{align}\label{penalty functions}
 \psi_\lambda(x)=\lambda(\|x\|_1 - f_\alpha(x)).
 \end{align}

The optimization problem (\ref{objective function}) is a nonconvex nonsmooth optimization  problems subject to the  penalty $\psi_\lambda(x)$. 
The cost function  $F$ given by   (\ref{objective function}) is in general  nonconvex nonsmooth. However, the convexity can hold under some conditions depending on 
   $A$, $\lambda$ and $\alpha$.
   Note that  the  main idea  of using such nonconvex penalty functions is to promote the sparsity of the solutions in (\ref{objective function}).  A  non-convex penalty can induce a nonconvex cost function, thus  unnecessary  suboptimal local minimizers for the cost function. The main goal of the present paper is twofold. First we introduce a class of  functions which majorize the cost function locally.  Then we use these majorizers (surrogaters) in an MM algorithm to 
  solve the optimization problem (\ref{objective function}) and prove that  the iteration points 
   convergence  to the stationary point of the objective function under some sufficient condition.   
  Before we explain the main contributions of the current work in details, let us first recall  some  known and special cases of  (\ref{objective function}).

 \vskip.124in 
 
 {\bf Special cases.} When 
  $\lambda=0$, the problem is alternately referred to
as minimizer of the residual sum of squared errors (RSS). The solution for minimization can be obtained by least square method. In this case, the  minimization is  continuously differentiable unconstrained convex optimization
problem. For a solution of this case, see e.g. 
(\cite{Hastie-Tibshi-Friedman01}).  
    When   $f_\alpha$ is a constant function (which happens in our case, for example,    when $\alpha=0$),   the problem turns into the classical $\ell_1$ regularizer case. 
This case among the cases with   convex regularizer (or penalty term) is more effective in inducing sparse solutions for (\ref{opt problem}) and (\ref{objective function})  (\cite{BED09}). However, the $\ell_1$ regularizer  underestimates the high amplitude components  of the solution. The least square problem with an $\ell_1$ penalty is known 
as {\it Least Absolute Selection and Shrinkage Operator}
(LASSO)  (\cite{Tibshi_LAsso}) and {\it Basis Pursuit Denoising} (\cite{Donoho99}),   respectively. Several methods have been introduced in \cite{Tibshi_LAsso,Donoho99} for optimizing the problem.  
   When $f(x)= \|x\|_1$, the  Moreau envelope  $f_\alpha$ is the well-known Huber function.  The Huber function   and its general form as  regulizers of sparse recovery problems  have been treated in \cite{IS17}, and it has been proved that  with these regularizers, using proximal algorithms,  the problem  (\ref{objective function})  has an optimal solution provided that $F$ is convex.    In this case, the penalty term 
   (\ref{penalty functions}) is called GMC penalty. 
   % % and the  minimization is done using 
%proximal algorithms comprising simple computations.
 
\vskip.124in

 {\bf Main contribution.} 
 The first contribution of this paper is  to construct a class of convex  functions which      majorize  (surrogate)   the cost  function   $F$  (\ref{objective function}) locally.  We obtain these  functions by  constructing  local  minimizers for the penalty term $\psi_\lambda$. The  local majorizers are tangent to the cost function only at one  point and each has a global minimum.  The  existence of a global minimum for the majorizers is obtained by convexity of majorizers, which we also study here.  %Each minimum point defines the iteration point for the next step in the algorithm.  order to obtain a minimum for the majorizers in any convex set, we obtain a sufficient condition for which the majorizers are convex.
%  Next we show that the directional derivative of the majorizers at the touching points are equal. 

 The second contribution  of this paper is to use the MM algorithm  to find  a sequence of    iteration  points  which converges to the local minimum of the cost function (\ref{objective function}). In this algorithm,  the initial point $x^{(0)}$  is taken away from zero and each   iteration point $x^{(k)}$ is obtained by local minimization  of surregator function $F^M(\cdot, x^{(k-1)})$  in some small neighbourhood of $x^{(k-1)}$.  We prove that the sequence $\{x^{(k)}\}_k$  has an accumulation point   and it is a  stationary point for the cost function $F$  in (\ref{objective function}), provided that  the majorizers are $a$-strongly convex.

\vskip.124in

 {\bf Outline.}  The paper is organized as follows. After introducing some notations and preliminaries in Section \ref{pre-note}, in Section \ref{Auxiliary results}  we introduce a class of minimizer functions for the penalty term (\ref{penalty functions}) and obtain majorizers for the cost function $F$ (\ref{objective function}). In this section, we also study sufficient   conditions for the majorizers to be  convex.  These results are collected in  
  %that for a fixed point $w$, the functions  $F^M(\cdot, w)$ and $F^m(\cdot,w)$ are local majorizer and minimizers of $F$ with equal directional derivative at $w$. The  result are proved 
Lemma  \ref{Majorizer and Minorizer of S-alpha}  and 
    Theorem \ref{local majorants for $F$}.  
  In Section \ref{iteration points} we  propose to use the iterative  MM algorithm with initial point away from zero  to guarantee  the  convergence of iteration points  to a  stationary point of $F$.  These results are collected in Proposition  
\ref{sequence converges}, Theorem \ref{stationary point}   and Corollary \ref{main theorem}.     %In Section \ref{iteration points} we  also discuss initial  points which induce a convergent subsequence to a stationary point. 
  
\subsection{Related work}
The current paper is proposing the use of majorization-minimization (MM) algorithm  to solve the class of nonconvex nonsmooth optimization problems of type (\ref{objective function}). This approach has been used for example in \cite{Figueir07,LangeChi14,LMSS16} for solving some optimization problems different than what we consider here. There are another types of methods that have been proved effective in solving nonconvex problems. For example, 
iteratively reweighted least squares (IRLS) method (\cite{DaubechiesDeVor10}) and iteratively reweighted $\ell_1$ (IRL1) (\cite{Candes08}). For a list of other methods, we refer the reader to see \cite{LMSS16} and the reference therein. 
 \subsection{Acknowledgement} The author wishes to thank Prof. Ivan Selesnick for several   insightful  discussions and for introducing her the MM algorithm. 
 
 \section{preliminaries and notations}\label{pre-note}

  \begin{figure}[t]
 \centering
  %\vspace{-140pt}
 \includegraphics[scale=1]{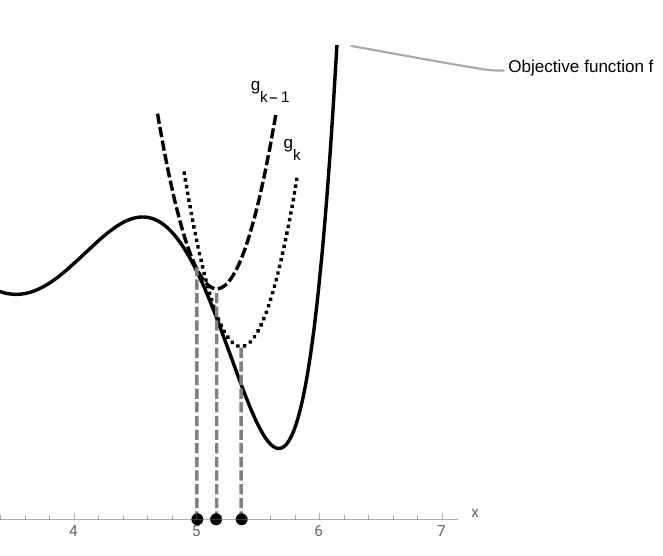}
% \includegraphics[width=2in]{MM1}
% \vspace{-50pt}
 \caption{\small{The MM algorithm procedure. The points represent the iteration points $x^{(k-1)}, x^{(k)}, x^{(k+1)}$.}} 
 \label{fig:MM algorithm}
 \end{figure}

 For any vector $x\in \Bbb R^N$, the $\ell_1$ and $\ell_2$ norms of $x$ are defined by 
 $\|x\|_1= \sum_i|x_i|$ and  $\|x\|_2^2= \sum_i|x_i|^2$.  By $A\in\Bbb R^M\times \Bbb R^N$ we denote the matrix of dimension $M\times N$ and we say   it  is semindefinite positive and denote it by $A\succeq 0$ if for all $x\in\Bbb R^N$,  $\langle Ax,x\rangle\geq 0$. Here, $\langle \ , \ \rangle$ denotes the  inner product of two vectors. The positive definite is also equivalent to say that all eigenvalues of $A$ are non-negative. 
 \vskip.124in 
  
  {\it Local majorizers and minimizers:}   Given a fixed point $w\in \Bbb R^N$, a function $g(\cdot, w): \Bbb R^N\to \Bbb R$ is called a local  majorizer of an objective function $f:\Bbb R^N\to \Bbb R$ at $w$ if the following conditions hold:
   \begin{align}\label{defn-majorizer}
   f(x)&\leq  g(x,w) \quad \forall x\in \Bbb R^N, \\\notag
   f(x) &= g(x,w) \quad  \text{if and only if } x=w .
   \end{align} 
   
   The functions $f$ and $g(\cdot, w)$ are tangent at the point $w$ when $f$ and $g(\cdot, w)$ both have directional  derivatives at $w$ and for any direction $d\in \Bbb R^N$ with small $ \|d\|_2$, 
$$\nabla g(w;d,w)= \nabla f_{\alpha}(w;d) .$$ 

From the  point of view of geometry, a majorizer means that the surface obtained by the map $x\mapsto  g(x,w)$ lies above the surface generated by  $x\mapsto f(x)$ and these two surfaces are touching and have tangent point only at $x=w$. 

We say a function $g(\cdot, w): \Bbb R^N\to \Bbb R$ minorizes the function $f$ at $w$
when $-g(\cdot, w): \Bbb R^N\to \Bbb R$ majorizes $-f$ at $w$. 
 \vskip.124in

  {\it Iterative MM algorithm.}   In minimization algorithm, we choose the majorizer $g_{k-1}:=g(\cdot,x^{(k-1)})$ tangent to objective function at $x^{(k-1)}$ and minimize it on a convex set $\mathcal D$ to obtain next iteration point 
    $x^{(k)}$.   That is $x^{(k)}:= \argmin_{x\in \mathcal D} g(x,x^{(k-1)})$, provided that  $x^{(k)}$ exists. We define $g_k:=g(\cdot,x^{(k)})$ (see Figure ~\ref{fig:MM algorithm}).
    
      When the minmization points $x^{(k)}$ exists, the following descending property holds: 
  \begin{align}\label{descenting properties} 
  f(x^{(k)})\leq g_{k-1}(x^{(k)})=g(x^{(k)},x^{(k-1)})\leq g_{k-1}(x^{(k-1)})=g( x^{(k-1)},x^{(k-1)})=f(x^{(k-1)}). 
  \end{align}
 
 One of the  significant properties of the MM algorithm is its stability due to the descending property of the objective function $f$ (\ref{descenting properties}). 
 If an objective function is strictly convex, then
the MM algorithm will converge to the unique optimal point (global minimum),
assuming that it exists. In the absence of  convexity, all stationary points are isolated, then the MM algorithm will converge to one of them. 
 For a complete philosophy of the MM algorithm, we refer the reader, for example,   to \cite{Lang-opt,LangeChi14}. 
 
 \vskip.124in 
 
 {\it Moreau envelope.} 
 For a function  $\tilde f:\Bbb R^N\to \Bbb R$ and  $\alpha>0$,  the Moreau envelope of the function $\tilde f$ is denoted by $f_\alpha$ and is  defined by    infimal convolution
  \begin{align}\label{Mourea envelop}  f_\alpha(x) := \left(\tilde f\small\Box \cfrac{\alpha}{2}\|\cdot\|_2^2 \right)(x)= \inf_{v\in \Bbb R^N} \{\tilde f(v)+\cfrac{\alpha}{2}\|v-x\|_2^2\},\quad  \forall x\in\Bbb R^N .  
   \end{align}
   The function $f_\alpha$ is convex when $f$ is convex and it is 
 the infimal convolution of the  function  $f$ and the map $x\mapsto \cfrac{\alpha}{2}\|x\|_2^2$. 
   For example, when $\tilde f(x) = \|x\|_1$,  the Moreau envelope  $f_\alpha: \Bbb R^N\to \Bbb R\cup \{\infty\}$ is well-known (generalized)  Hubber function. 
       For the definition of infimal convolution and its other properties  see,  e.g.,  \cite{BauscheCombettesBook}.  
       
        \vskip.124in 
        
     Let $y\in \Bbb R^M$ be an observed vector data and     $A\in \Bbb R^M\times \Bbb R^N$  be a matrix, which is usually a wide low rank matrix. 
    The following result for cost functions $F$  given in   (\ref{objective function}) with penalty term  $\psi_\lambda (x)= 
    \lambda(\|x\|_1 - f_\alpha(x))$   is 
 a mild improvement of   Theorem 1 in \cite{IS17}. In  \cite{IS17},  $\psi_\lambda $ is the GMC penalty and the Moreau envelope  $f_\alpha$ is the  generalized Huber function. 
 
\begin{theorem} The function $F$ is (strictly) convex if  the convexity condition
\begin{align}\label{convexity condition1} 
A^TA-\lambda \alpha I \succeq 0,
 \end{align}
 holds.  For  strictly convex the inequality $\succeq 0$ is replaced by $\succ 0$.  Here,  $I$ is the identity matrix.
 \end{theorem} 
 The proof of this theorem can be obtained by a similar technique which was used to   prove Theorem 1 in \cite{IS17}. 
 Note that the condition (\ref{convexity condition1})  ensures the  uniqueness of the minimizer of cost function $F$.   
        \vskip.124in
        
 The convexity condition (\ref{convexity condition1})  indicates  that  all eigenvalues of  matrix $A^TA$ must be  at least $\lambda \alpha$. 
 In the absence of convexity condition   (\ref{convexity condition1}), the function $F$ is  as sum of a concave function with a convex function and  may not be convex, and therefore it  may not have any local  minimum. In this case,  one needs an approach   to prove the existence of  a global minimum or global optimum point for $F$.  This paper proposes the  use of  MM algorithm technique to  solve the minimization problem for $F$, when $F$ is nonconvex.  
  \vskip.124in

  To reach our goal and prove the existence of a local minimizer for nonconvex  objective (or cost)  function $F$ (\ref{objective function}),  we first  construct local minimizers  for the  Moreau envelope $f_\alpha$,  and then use them to obtain local majorizers for $F$.
    Let $\gamma_m>0$ be a constant to be determined later.   
    %For any    $w\in \Bbb R^N$ %,  we can write 
   %
% \begin{align}\label{defn of maj and min}
 %f_\alpha(x)\geq  f_{\alpha}(x)-\gamma_m \|x-w\|_2^2.   %\leq f_{\alpha}(x)+\gamma_M \|x-w\|_2^2. 
% \end{align} 
For any    $w\in \Bbb R^N$, let %$f_{\alpha}^m(\cdot, w)$  as 
 \begin{align}\label{defn of maj and min1}f_{\alpha}^m(x, w):= f_{\alpha}(x)-\gamma_m \|x-w\|_2^2  . 
 \end{align}
We define   $F^M(\cdot,w)$ by replacing $f_\alpha$ by $f_{\alpha}^m(x, w)$ in the definition of the objective function $F$ (\ref{objective function}) as follows: 
 \begin{align}\label{majorizers}F^M(x,w): = \cfrac{1}{2} \|y-Ax\|_2^2 + \lambda\left(\|x\|_1- f_{\alpha}^m(x,w)\right). 
 \end{align}
 It is obvious that for all $x\in\Bbb R^N$, $F^M(x,w)\geq F(x)$. In Theorem \ref{local majorants for $F$} we prove that the surface generated by the function $F^M(\cdot, w)$ is lying about  the surface generated by the function $F$ and they touch only at one  point $x=w$.  
 %and 
 
  %$$F^m(x,w): = \cfrac{1}{2} \|y-Ax\|_2^2 + \lambda\left(\|x\|_1- f_{\alpha}^M(x,w)\right)
 %$$
\begin{remark} 
In the same fashion,  one can define minorizers $F^m(x,w)$ for $F$. For this, let $\gamma_M>0$ and define 
 $f_{\alpha}^M(x, w):= f_{\alpha}(x)+\gamma_M \|x-w\|_2^2$. Then $f_{\alpha}^M(x, w)\geq f_\alpha(x)$ for all $x$.   Define 
 $$F^m(x,w): = \cfrac{1}{2} \|y-Ax\|_2^2 + \lambda\left(\|x\|_1- f_{\alpha}^M(x,w)\right). 
 $$
 With a similar techniques of proofs for majorizers in the rest of the present paper, one can obtain local
  minorizers for the cost function $F$ with tangential point at $w$.  This is a useful tool in finding local maximums of an optimization problem. 
\end{remark} 
 
    \section{Construction of a local majorizer  for   cost function}\label{Auxiliary results} 
   Our first result in this section  proves the existence of local   minorizers for the Moearu envelope function 
  $f_\alpha$. 
   \begin{lemma}[Minorizer of $f_\alpha$]\label{Majorizer and Minorizer of S-alpha} Fix $w\in \Bbb R^n$   and define    
  $f_{\alpha}^m(\cdot,w)$ as  in (\ref{defn of maj and min1}).  $f_{\alpha}^m(\cdot,w)$ is a minorizer for  $f_\alpha$, and  for any   direction $d\in \Bbb R^N$ with $\|d\|_2$ small, we have 
 
 \begin{align}\label{tangential property}\nabla f_{\alpha}(w;d)=\nabla f_{\alpha}^m(w;d,w).   
 \end{align} 
  
 \end{lemma}
 
 \begin{proof} The  proof of the  local miniorizers for $f_\alpha$ is obtained directly from the definition of $f^m(\cdot, w)$. To prove  (\ref{tangential property}), let $d\in \Bbb R^N$ with $\|d\|_2$ small. Then

\begin{align} \nabla f_\alpha^m(w;d,w)& =   \liminf_{\theta\to 0^+} \frac{f_{\alpha}^m(w+\theta d,w)- f_{\alpha}^m(w,w)}{\theta} \\\notag
& = \liminf_{\theta\to 0^+} \frac{f_{\alpha}^m(w+\theta d,w)- f_{\alpha}(w)}{\theta} \\\notag
&= \liminf_{\theta\to 0^+} \frac{\left(f_{\alpha}(w+\theta d)-\gamma_m \| \theta d\|_2^2 \right)
- f_\alpha(w)}{\theta}   \\\notag
&= \liminf_{\theta\to 0^+} \frac{\left(f_{\alpha}(w+\theta d) 
- f_\alpha(w)\right) -\gamma_m \| \theta d\|_2^2}{\theta}   \\\notag
&=  \liminf_{\theta\to 0^+} \frac{\left(f_{\alpha}(w+\theta d) 
- f_\alpha(w)\right)  }{\theta}   - \liminf_{\theta\to 0^+} \gamma_m \theta \|d\|_2^2  \\\notag
&= \liminf_{\theta\to 0^+} \frac{\left(f_{\alpha}(w+\theta d) 
- f_\alpha(w)\right)  }{\theta}  \\\notag
&= \nabla  f_\alpha(w;d). 
\end{align}
 This completes the proof of the  theorem. 
 \end{proof} 
 
% {\bf Remark 1.} The results of the previous theorem also holds for any constants  $\gamma_m$ and $\gamma_M$ and any fixed point $w$ instead of $x^{(k)}$.   \\ 
  
  % {\bf Remark 2.} The Huber function $S_\alpha$ in the previous theorem can be replaced by any function and the results will still hold true. \\ 
   
Our next result illustrates that  the  local  minorizers  of the Moreau envelope function $f_\alpha$ induce local majorizers for 
  $F$.     
 
 \begin{theorem}\label{local majorants for $F$}
The function $F^M(\cdot,w)$ (\ref{majorizers})  is  local  majorizer   for the cost function $F$ at $w$, and we have 
 \begin{itemize} 
  \item[(i)] $\nabla F(w;d)=\nabla F^m(w;d,w) $ for all $d$  with $\|d\|_2$ sufficiently small. 
  \item[(ii)] $F^M(\cdot,w)$ is (strictly) convex if 
\begin{align}\label{convexity condition} 
A^TA + \lambda (2\gamma_m -\alpha) I \succeq 0. 
\end{align} 
  %and accordingly, the  function $F_k^m$ is convex if 
  %$$A^TA + (2\lambda \gamma_M -\alpha) I \succeq 0$$ 
  \end{itemize}
 \end{theorem}
 
 \begin{proof} By Proposition \ref{Majorizer and Minorizer of S-alpha} it is immediate that  the function $F^M(\cdot, w)$ is a local majorizer for $F$. The item (i) also holds by (\ref{tangential property}).  To prove the item (ii), 
we will adapt an approach used   to prove  Theorem 1 in  \cite{IS17}. 

Notice  the discrepancy with respect to the data in the surregator function
   $F^M(\cdot, w)$ can be written as 
    \begin{align}\label{expansion}
    F^M(x,w) = x^T\left(\frac{1}{2}A^TA+\lambda(\gamma_m  -\frac{\alpha}{2})I\right)x +\lambda\|x\|_1+ \max_{v\in \Bbb R^N} g(v,x,w).  
    \end{align}
    
Notice   
the function $Q(x):= \max_{v\in \Bbb R^N} g(v,x,w)$ is not affine  although for any 
 fixed $(v,w)$, the  map $x\to g(v,x,w)$ is an  affine (or linear) function.  However, the convexity of $Q$ can be obtained as a result of Proposition
8.14 in \cite{BauscheCombettesBook}, since $Q$ 
   is  pointwise maximum of convex functions. Therefore by  (\ref{expansion}),    $F^M(\cdot,w)$ is  convex when the quadratic part is convex. This means when the  matrix $A^TA+\lambda(2\gamma_m  -\alpha)I$ is  positive definite and this completes the proof of (ii).  The majorizer function is strictly convex when the inequality is strict. 
 \end{proof}

% \begin{example} The constant $\alpha$ plays the roll of    convexity controller  for the penalty function is given.   When $\alpha$ is to determine along the constant $\gamma_m$, then one must choose these constants such that 
 
 %$$2\gamma_m -\alpha\geq  \frac{\kappa_0^2}{\lambda}.$$
 %It is trivial that there are infinity many solutions for the pairs $(\gamma_m,\alpha)$ satisfying the preceding inequalities. 
% For example, if $0\in \Lambda$, the line $y=2x$ crossing the origin is a subset of solution set for the pairs $(x,y)=(\gamma_m,\alpha)$. 
%  \end{example}
  
 \section{MM algorithm and stationary points}\label{iteration points}
 
 In this section,   we prove the existence of a sequence of iteration points which are obtained by minimizing    surregator functions at each iteration step. 
 Under   strongly convexity  condition for the surregators we show that 
  the  iteration points have a convergent   subsequence  and the limit point is a   stationary point of $F$.  The stationary point is local minimum for $F$ by  the descending property (\ref{descenting properties}).  To prove the existence of the sequence, we continue as follows. First we introduce  a notation. For $\epsilon>0$ and 
 $u\in \Bbb R^N$, we denote by $B_\epsilon(u)$  the ball of radius $\epsilon$ with respect to the $\ell_2$ norm with center $u$. That is,   the set of all  points $x\in \Bbb R^N$ with $\ell_2$ norm distance    from   the center $u$ less than $\epsilon$. 
 
 \begin{proposition}\label{sequence converges}  Let $\alpha>0$ and 
   $\epsilon>0$. 
     Then the sequence obtained by the following iterative algorithm converges. 
     
     \begin{algorithm}[H]
 \SetAlgoLined
%\KwResult{Write here the result }
 Set $\gamma_m$  such that  the  convexity condition (\ref{convexity condition}) holds\;
Initialize 
    $x^{(0)}\in \Bbb R^N$ such that  $\|x^{(0)}\|_2>2 \epsilon$\;
 \For{$k=0, \cdots ,$}{
  $x^{(k+1)}=\argmin_{x\in B_{\frac{\varepsilon}{2^k}}(x^{(k)})} F^M(x,x^{(k)})$\;
 }
 %\caption{How to write algorithms}
\end{algorithm}
     where $k$ is the iteration counter. 
 \end{proposition} 
     \begin{proof} To prove the  proposition, first we claim that 
the sequence $\{ x^{(k)}\}_k$ has a convergent subsequence. Then we show that the sequence $\{ x^{(k)}\}_k$  is that   subsequence. \\ 

 {\it (Boundedness)} 
The  iteration points $x^{(k)}$ satisfy   
 \begin{align}\label{difference of iteration points}
 \|x^{(k+1)} -x^{(k)}\|_2\leq  \frac{\varepsilon}{2^k}, \ \  \forall k\geq 0.  
 \end{align} 
 Using this, iteratively one can show that 
  for any $k$ 
\begin{align}\label{some ineq}
0<\|x^{(0)}\|_2 - \varepsilon \leq \|x^{(k)}\|_2.  
\end{align}
  
  Indeed, the the left side of (\ref{some ineq}) is a positive constant since 
  the initial point $x^{(0)}$ is chosen such that  $0<\varepsilon < \frac{ \|x^{(0)}\|_2}{2}$.     From the other hand,  the relation (\ref{difference of iteration points}) implies that the sequence is also bounded above.   Therefore,  by The Bolzano-Weierstrass Theorem the 
 $\{ x^{(k)}\}_k$ has a convergent subsequence with accumulation point $x^\ast$. In what follows we prove that the sequence  $\{ x^{(k)}\}_k$ converges to  
   $x^\ast$.  \\

  {\it (Convergence)} 
 Assume $\{x^{(k_n)}\}_n$ be a subsequence of  $\{ x^{(k)}\}_k$ such that $x^{(k_n)}\to x^*$ as $k_n\to \infty$. Fix $k$ and let $k_n>k$. An easy calculation shows that 
  $$\|x^{(k)}-x^*\| \leq \|x^{(k_n)}-x^\ast\|  
  + \mathcal{O}(\frac{\epsilon}{2^{k_n}})\quad \text{as} \ k\to \infty .$$ 
  %Since the sequence is bounded from below, the accumulation point is a limit point of a subsequence of $\{ x^{(k)}\}_k$ and t
This implies that $x^\ast$ is the accumulation point for  $\{ x^{(k)}\}_k$ and we are done.   
  \end{proof} 
  
  Notice the limit point may not be  a stationary    (local minimum) point.   However, this can be   obtained under some sufficient assumptions on the majorizers.  
First we have a lemma. 
  %\begin{definition}[a-strongly convex]  ???? 
  %\end{definition} 
%  The $a$-strongly convexity condition is equivalent to say that 
 %all eigenvalues of the matrix  $A^T A + \lambda(2\gamma_m- \alpha)I$ must be at least $a$.     

 \begin{lemma} Let $a>0$ and $w\in\Bbb R^N$. The local majorizer 
 $F^M(\cdot,w)$ is $a$-strongly convex provided   that  $A^T A + \lambda(\gamma_m- \alpha)I\succeq 2a I$. 
 \end{lemma} 
  
 \begin{proof}  Recall the discrepancy of data given in   (\ref{expansion})  
  $$F^M(x,w) = x^T\left(\frac{1}{2}A^TA+\lambda(\gamma_m  -\frac{\alpha}{2})I\right)x +\lambda\|x\|_1+ \max_{v\in \Bbb R^N} g(v,x,w) .$$
This representation implies that 
 $F_k^M$ is $a$-strongly convex when 
  \begin{align}\label{strong convexity condition}
  \frac{1}{2}A^TA+\lambda(\gamma_m  -\frac{\alpha}{2})I  \succeq al, 
  \end{align}
and we are done. 
   \end{proof}  
 
     Strong convexity is one of the most important tools in optimization and in particular    it guarantees  linear convergence rate of many gradient decent based algorithms. 
Here, we recall a result: 
  
   \begin{lemma}[\cite{Marial13}, Lemma B.5]\label{Marial13}
 Let $f$ be an $a$-strongly convex on a convex domain $\mathcal D$. Let $x^*$ be the minimizer of $f$ on $\mathcal D$. Then  
  $$a\|x-x^*\|_2^2\leq f(x)-f(x^*)  \quad \forall x\in \mathcal D .$$ 
 \end{lemma}

 As an outcome of the  lemma we prove that the limit point $x^*$   in  Theorem \ref{sequence converges} is a stationary point (thus a local minimizer) for $F$: 
  
 \begin{theorem}\label{stationary point}  Assume that $a$-strongly convexity condition 
 (\ref{strong convexity condition}) hold, and 
 $\{x^{(k)}\}$ converges to   $x^\ast$. Then  $x^\ast$ is a stationary point for $F$ and we have $\nabla F(x^\ast;d)\geq 0$.
 \end{theorem} 
 \begin{proof} For all $x\in \Bbb R^N$, $f_{\alpha}^m(x,x^{(k)}) \to f_\alpha(x) -\gamma_m \|x-x^\ast\|_2^2$ as $k\to \infty$. Thus, $F_k^M(x) \to F(x) +\lambda \gamma_m \|x-x^\ast\|_2^2$,  $k\to \infty$. 
From the other side, by applying Lemma \ref{Marial13} to $F_k^M$ and using the majorization property  of this function we obtain  
  $$a\|x-x^{(k+1)}\|_2^2 
\leq F^M(x,x^{(k)})-F^M(x^{(k+1)},x^{(k)}) 
\leq F^M(x,x^{(k)})-F(x^{(k+1)}) \quad \forall x\in \Bbb R^N.$$
%(For convenience, we can take $\mathcal D$ a small neighbourhood of $x^\ast$.)
So, 
$$a\|x-x^{(k+1)}\|_2^2 
\leq  F^M(x,x^{(k)})-F(x^{(k+1)}) \quad \forall x\in \Bbb R^N.$$

By the continuity of  $F$,  
by letting   $k\to \infty$ in the preceding   inequality,   we obtain  
 
 $$a\|x-x^\ast\|_2^2 \leq F(x) + \lambda \gamma_m \|x-x^\ast\|_2^2 -F(x^\ast),$$ 
 or equivalently 
\begin{align}\label{some ineq}
  F(x)  -F(x^\ast) \geq   (a-\lambda \gamma_m) \|x-x^\ast\|_2^2.  
  \end{align}
  Let $d\in \Bbb R^N$ be a direction with 
  $\|d\|_2\leq \epsilon$  and  $\theta>0$. By (\ref{some ineq}),   
 
  $$ F(x^\ast +\theta d)  -F(x^\ast) \geq 
(a-\lambda \gamma_m)  \theta^2 \|d\|_2^2.$$
 
 This implies that 
 
 $$\nabla F(x^\ast;d)=\liminf_{\theta\to 0^+} \cfrac{F(x^\ast +\theta d)  -F(x^\ast)}{\theta} \geq 
(a-\lambda \gamma_m)   \|d\|_2^2 (\liminf_{\theta\to 0^+} \theta)   = 0, $$

and we are done.    
 \end{proof}

 The following result is a summary of the results presented in  this and previous sections. 
  
  \begin{corollary}[Convergence]\label{main theorem}    Assume that the local majorizers  $\{F^M(\cdot,x^{(k)})\}_k$  of $F$ are  $a$-strongly convex. The sequence  of iteration points  $\{x^{(k)}\}$ converges and the limit point is a  local minimizer of $F$. 
 \end{corollary} 
 
  \begin{proof}  By Theorem   \ref{stationary point}, $\nabla F(x^\ast;d) \geq 0$, thus  $x^\ast$ is an stationary point. By the descending property (\ref{descenting properties}), the stationary point is a local minimum. 
  \end{proof}

  We conclude  this section  by illustrating some examples. % for the choice of $\gamma_m$ to ensure the convexity and $a$-strong convexity of majorizer functions $F^M$, while the convexity condition (\ref{convexity condition1}) fails for the objective $F$. 
 First we have a notation. For a given matrix $A$, we denote by $\Sigma(A)$ the set of all singular values of matrix $A$. 
  
 \begin{example}[Tight frame] Assume that the rows of matrix $A$ form a tight frame with frame constant $C$. Then    $A^TA=CI$   and  $\Sigma(A):=\{C\}$. (When $C=1$,  the rows of matrix $A$ form a normalized tight frame, also known as Parseval frame.)       Let $\alpha$ and $\lambda$ are given such that    $\alpha>C\lambda^{-1}$.   Then the sufficient convexity  condition (\ref{convexity condition1}) fails for  $F$ and the function $F$ may have no local minimum.  
 \end{example}

In the following   example  we present a  positive lower bound for $\gamma_m$ for which the convexity condition (\ref{convexity condition})  holds for the majorizers.   
   
 \begin{example} Assume that the convexity condition (\ref{convexity condition1}) fails. 
 Thus, 
    for some   
  $\sigma\in \Sigma(A)$ we must have  $\alpha >  \frac{\sigma}{\lambda}$.  This implies that for smallest singular value $\sigma_0$ we also have $\alpha > \frac{\sigma_0}{\lambda}$. 
  Define $c:=  \frac{\lambda\alpha-\sigma_0}{2\lambda}$. The constant $c$ is positive  and with a straightforward computation one can show that for  all pairs $(\gamma_m,a)$   satisfying   
\begin{align}\label{lower bound for gamma}\notag
\gamma_m    \geq \frac{a}{\lambda} +c 
\end{align}
%where $c$ is a constant and given by $c:=  \frac{\lambda\alpha-\sigma_0}{2\lambda}$.  
%With a straightforward computation one can show that for  $\gamma_m$ and $a$ satisfying  (\ref{lower bound for gamma}) 
the  $a$-strong convexity condition   
   holds for 
  surregators $F^M$. 
 %When $a=0$, the strong convexity turns into convexity and for $\gamma_m$ with 
  %  $$\gamma_m\geq   c$$ 
% the convexity condition  (\ref{convexity condition}) holds.  
 The strict convexity also holds when the inequality is strict. 
 \end{example}

\end{document}